\documentclass[11pt]{amsart}
\usepackage[hmargin=3cm,vmargin=3cm]{geometry}

\usepackage{combelow} 

\usepackage[latin1,utf8]{inputenc} 
\usepackage[english]{babel}
\usepackage[backend=bibtex,doi=false,url=false,isbn=false,style=alphabetic,maxnames=6]{biblatex}
\usepackage{etex}
\AtBeginBibliography{\small}

\usepackage{amsthm,amsmath,amssymb,amscd,amsfonts}
\usepackage{euscript,wasysym,xspace,stmaryrd,latexsym,youngtab,verbatim,mathrsfs}
\usepackage{palatino}
\usepackage{euscript}
\usepackage{graphicx}
\usepackage[all]{xy}
\SelectTips{cm}{}

\usepackage{tikz, tikz-cd}
\usetikzlibrary{matrix, arrows, calc}
\tikzset{frontline/.style={preaction={draw=white,-,line width=6pt}},}  

\usepackage[dvips]{epsfig}
\usepackage{pinlabel}
\usepackage{psfrag}

\IfFileExists{comments.sty}{\usepackage{comments}}

\RequirePackage{color}
\definecolor{references}{rgb}{0,0,1}

\RequirePackage[pdftex,
colorlinks = true,
urlcolor = references, 
citecolor = references, 
linkcolor = references, 
]
 {hyperref}

\newtheorem{thm}{Theorem}[section]

\newtheorem{lemma}[thm]{Lemma}
\newtheorem{theorem}[thm]{Theorem}

\newtheorem{proposition}[thm]{Proposition}

\newtheorem{corollary}[thm]{Corollary}

\newtheorem*{prop*}{Proposition}
\newtheorem*{lemma*}{Lemma}

\theoremstyle{definition}

\newtheorem{definition}[thm]{Definition}

\newtheorem{example}[thm]{Example}

\newtheorem{remark}[thm]{Remark}

\theoremstyle{remark}
\newtheorem{question}[thm]{Question}

\newtheorem{observation}[thm]{Observation}

\numberwithin{equation}{section}

  \def\ng{{\mathfrak n}}

    \def\KC{{\mathcal{K}}}

\def\AS{{\EuScript A}}

\def\CS{{\EuScript C}}
\def\DS{{\EuScript D}}

\def\RS{{\EuScript R}}

\def\YS{{\EuScript Y}}

\def\a{\alpha}
\def\b{\beta}

\def\d{\delta}

\def\e{\varepsilon}

\let\phi=\varphi
\let\tilde=\widetilde

\usepackage{bbm}
\def\C{{\mathbbm C}}
\def\N{{\mathbbm N}}

\def\Z{{\mathbbm Z}}

\def\1{\mathbbm{1}}

\renewcommand{\k}{\mathbbm{k}}

\newcommand{\smMatrix}[1]{\left[\begin{smallmatrix}#1\end{smallmatrix}\right]}

\newcommand{\sqmatrix}[1]{\left[\begin{matrix} #1\end{matrix}\right]}

\newcommand{\ip}[1]{\langle #1\rangle}

\newcommand{\refequal}[1]{\xy {\ar@{=}^{#1}
(-1,0)*{};(1,0)*{}};
\endxy}

\newcommand{\Hom}{\operatorname{Hom}}

\newcommand{\End}{\operatorname{End}}

\newcommand{\id}{\operatorname{id}}

\newcommand{\Id}{\operatorname{Id}}

\newcommand{\inv}{^{-1}}

\newcommand{\Ch}{\operatorname{Ch}}

\newcommand{\Cone}{\operatorname{Cone}}

\begin{document}

\newcommand{\Sq}{\operatorname{Sq}}
\newcommand{\summand}{\buildrel\oplus\over\subset}
\newcommand{\prefix}{pre}
\newcommand{\tw}{\operatorname{tw}}

\newcommand{\ti}{\mathtt{i}}
\newcommand{\tj}{\mathtt{j}}
\newcommand{\tk}{\mathtt{k}}
\newcommand{\Tw}{\operatorname{Tw}}
\newcommand{\cTw}{\operatorname{cTw}}
\newcommand{\pretr}{\operatorname{PreTr}}
\renewcommand{\e}{\epsilon}

\renewcommand{\ll}{\llbracket}
\newcommand{\rr}{\rrbracket}
\newcommand{\dgmod}{\text{-}\mathrm{dgmod}}
\renewcommand{\Id}{\operatorname{id}}

\title{Homological perturbation theory with curvature}

\begin{abstract}
We prove a general version of the homological perturbation lemma which works in the presence of curvature, and without the restriction to strong deformation retracts, building on work of Markl.  A key observation is that the notion of strong homotopy equivalence of complexes (or objects in an abstract dg category) has a natural expression in the language of curved twisted complexes.
\end{abstract}

\author{Matthew Hogancamp} \address{Northeastern University}

\maketitle

\setcounter{tocdepth}{1}
\tableofcontents

\newcommand{\she}{\RS_{s.h.e.}}
\renewcommand{\e}{\epsilon}
\newcommand{\twc}{\operatorname{cTw}}

\section{Introduction}
\label{s:intro}

\subsection{Homological perturbation theory}
\label{ss:intro curved hpt}
Much of homological algebra concerns the problem of transferring information (or structure) from a complex $(X,\d)$ to a complex $(X,\d+\a)$. For instance in the setting of abelian categories, the primary aim of spectral sequences is to relate the homologies of $(X,\d)$ and $(X,\d+\a)$ under various circumstances.

In particular, the subject of homological perturbation theory concerns the following problem: given a homotopy equivalence of complexes $(X,\d_X)\simeq (Y,\d_Y)$ and a ``perturbed'' differential $\d_X+\a$ on $X$, construct a perturbed differential $\d_Y+\b$ and a homotopy equivalence $(X,\d_X+\a)\simeq (Y,\d_Y+\b)$.  Note that this requires some hypotheses; for instance we could always take $\a=-\d_X$, and there are very few homotopy equivalences involving $(X,0)$. 

Homological perturbation theory is an essential tool in homological algebra and has a long history, going back to Eilenberg-Maclane \cite{EilMacGroupsI}.  Traditionally, homological perturbation theory is only discussed in the context of strong deformation retracts (see for instance \cite{GugLamPertI,GugLamPertII} and the references therein), but Markl \cite{MarklIdeal} has shown how this restriction can be eliminated by introducing the notion of strong homotopy equivalence.  This is a very useful observation, as every homotopy equivalence can be promoted to a strong homotopy equivalence (we give an elementary proof of this fact in Proposition \ref{prop:he iff she}).

In this paper we are concerned with the problem of homological perturbation theory for complexes with nonzero curvature, that is to say, complexes for which $\d^2\neq 0$.  We give some motivation for the appearance of curved complexes below.   Let us mention that a curved version of the homological perturbation lemma can be found in \cite{Tu-MFviaK}, but with the restriction to strong deformation retracts.  Without the crutch of strong deformation retracts, curved homological perturbation is quite a tricky matter.  Nonetheless it turns out that strong homotopy equivalence is exactly what is needed for our curved homological perturbation result (which specializes to Markl's Ideal Perturbation Lemma \cite{MarklIdeal} in the case of zero curvature, see \S \ref{s:curved hpl}).

\begin{remark}
We should also point out the unpublished notes of Jeff Hicks \cite{HicksCurvedAinfty-pp}, in which one finds a result concerning the transfer of curved $A_\infty$ algebra structures along a homotopy equivalence.  We do not see a direct relation with what is presented here, because in essence we fix a very boring curved $A_\infty$ algebra, namely $A=\k[z]$ with curvature $z$ and zero differential, and consider perturbation theory for curved dg $A$-modules (though we do not adopt this language in the sequel).
\end{remark}


\subsection{Why curved complexes?}
\label{ss:why curved}
In mathematics one often studies objects (topological spaces, schemes, etc) in families, and the same ought to be true of dg categories.  Many familiar dg categories (such as categories of complexes) do naturally come in families, and one way to construct such families involves the consideration of \emph{curved} complexes, as we now explain.

Let $\CS=\Ch^b(\AS)$ be the dg category of complexes over a $\k$-linear category $\AS$, and let $z$ be an element of the center of $\AS$ (i.e.~a natural endomorphism of the identity functor).  For instance we could take $R$ to be a $\k$-algebra, $\AS=R\text{-mod}$, and $z\in Z(R)$.  Let $y$ denote a formal indeterminate of degree 2.   Let $\YS$ denote the category whose objects are pairs $(X,\d(y))$ where $X$ is a finite sequence of objects in $\AS$ and
\[
\d(y) = \sum_{l\geq 0} \d_l y^l,\qquad\qquad \d_l\in \End^{1-2l}_\CS(X)
\]
is a formal sum of endomorphisms of $X$ with
\[
\d(y)^2=z_X y,
\]
where $z_X\in \End^0_\CS(X)$ denotes the action of $z$ on $X$ (i.e.~the componentwise action of $z$ on the chain objects $X^k\in \AS$).  From the boundedness assumption on $X$, the endomorphisms $\d_l$ are zero for all but finitely many $l\in \Z$, so $\d(y)$ is really a polynomial in $y$.  The pair $(X,\d)$ is a particular kind of \emph{curved complex} with \emph{curvature} $z_X y$.  These form a dg category, just as in the case of ordinary complexes (see \S \ref{ss:curved tw}).

\begin{remark}
Writing the equation $\d(y)^2=z_X y$ in terms of components, we see that $\d_0^2=0$, $\d_1\in \End^{-1}_\CS(X)$ is a homotopy for $z_X$, and $\d_l$ for $l\geq 2$ are certain higher homotopies.
\end{remark}

\begin{remark}
There is a dg functor $\YS\rightarrow \Ch^b(\AS)$ which sets $y=0$.   But we can also specialize $y$ to be any other element $t\in\k$ (this has the affect of collapsing the $\Z$ grading to a $\Z/2$ grading), obtaining a 1-parameter family of dg categories $\YS_t$ ($t\in \k$) (though perhaps it is better to regard $\YS$ as a sheaf of dg categories on affine space $\mathbb{A}^1_\k$).
\end{remark}

\begin{remark}
Any finite collection of central elements $z_1,\ldots,z_r$ in $\AS$ determines in a similar fashion a sheaf of dg categories on $\mathbb{A}_\k^r$, by considering the dg category of complexes with curvature $\sum_i  z_i y_i$.
\end{remark}

Such families can often shed light on the original category $\CS=\Ch^b(\AS)$ (often by means of spectral sequences from hom complexes in $\CS=\YS_{t=0}$ to hom complexes in nearby categories $\YS_{t\neq 0}$).  One successful instance of this (and our primary motivation for considering curved complexes) is the main result of \cite{GorHog-y-pp}, which relates the Hecke category in type $A$ to the isospectral Hilbert scheme of points in $\C^2$.

\begin{remark}
The notions of curved dg algebras and modules also appear naturally in symplectic topology and Fukaya categories (see \cite{AurouxBeginners}, Remark 2.12).  Additionally, matrix factorizations \cite{Eis80} can also be regarded as certain kinds of curved complexes.
\end{remark}

\subsection{Outline of the paper}
\label{ss:outline}
The remainder of the paper is written using the language of dg categories, the basics of which we recall in \S \ref{s:basics}.  Specifically, \S \ref{ss:dgcats} introduces dg categories, functors, and natural transformations.  In \S \ref{ss:suspension} we recall the notion of additive suspended envelope $\Sigma \CS$.  In \S \ref{ss:twists} we recall the notion of twists in a dg category (which is an abstraction of the relationship between complexes $(X,\d)$ and $(X,\d+\a)$).  In \S \ref{ss:cones} we recall the mapping cone construction and recall the standard result that cones detect homotopy equivalences. In \S \ref{ss:twists of contractible} we recall a standard result concerning twists of contractible objects in dg categories.

In \S \ref{ss:curved tw} we introduce the notion of curved twisted complexes over a dg category $\CS$.  In \S \ref{ss:she} we introduce the notion of strong homotopy equivalence and explain how this can be expressed quite naturally in the language of curved twisted complexes.  We also prove several very useful results, including how to promote a homotopy equivalence to a strong homotopy equivalence (Proposition \ref{prop:he iff she}).  This uses the technical but useful fact that the data of a strong homotopy equivalence $X\simeq Y$ is equivalent to the data of a strong homotopy equivalence $\Cone(X\rightarrow Y)\simeq 0$ (Lemma \ref{lemma:she iff cone sc}).

\S \ref{s:curved hpl} contains all the results on perturbation theory.  The main result (Theorem \ref{thm:curved hpl}) on curved homological perturbation theory is stated and proved in \S \ref{ss:main result}, and \S \ref{ss:consequences} develops useful corollaries of this main result (including Markl's Ideal Perturbation Lemma, stated as Corollary  \ref{cor:uncurved hpl}).

\subsection*{Acknowledgements}
We owe an intellectual debt to Martin Markl, from whose paper \cite{MarklIdeal} we first learned of strong homotopy equivalence, and whose Ideal Perturbation Lemma turned out to be the key to our main result.  We also thank Florian Naef for the long and enlightening discussions, and for tolerating the author's  (not so) recent fascination with curved complexes, and Paul Wedrich for comments on an earlier draft.

The author was supported by NSF grant DMS 1702274.

\section{DG categories}
\label{s:basics}

\subsection{First definitions}
\label{ss:dgcats}

\subsubsection{DG $\k$-modules}
\label{sss:dg kmod}
Let $\k$ be a commutative ring, and let $\CS(\k)$ be the category of complexes of $\k$-modules, with the cohomological convention for differentials:
\[
\cdots \buildrel d\over \rightarrow X^k \buildrel d\over \rightarrow  X^{k+1}\buildrel d\over \rightarrow \cdots.
\]
Objects in $\CS(\k)$ are also called dg $\k$-modules.

Morphism spaces in $\CS(\k)$ are by definition the complexes $\Hom_{\CS(\k)}(X,Y)$ which in degree $k$ are defined by
\[
\Hom_{\CS(\k)}^k(X,Y) = \prod_{i\in \Z}\Hom_\k(X^i,Y^{i+k}),
\]
with differential given by the ``\emph{super-commutator}''
\[
f\mapsto [\d,f]:= \d_Y\circ f - (-1)^{|f|} f\circ \d_X,
\]
where $\d_Y$ and $\d_X$ are the differentials on $Y$ and $X$, and $|f|\in \Z$ denotes the degree of $f$.

The category $\CS(\k)$ is monoidal with the usual tensor product of complexes
\[
(X\otimes_\k Y)^k:=\bigoplus_{i+j=k} X^i\otimes_\k Y^j,\qquad\qquad \d_{X\otimes Y} = \d_X\otimes \Id_Y + \Id_X\otimes \d_Y.
\]
The tensor product of morphisms (needed in order to make sense of the above expression of $\d_{X\otimes Y}$) is defined using the usual sign rule
\[
(f\otimes g)(x\otimes y) = (-1)^{|g||x|}f(x)\otimes g(y).
\]
This sign rule guarantees that the composition and tensor product are related by
\[
(f\otimes g)\circ (f'\otimes g') = (-1)^{|g||f'|} (f\circ f')\otimes (g\circ g')
\]
which, in turn, guarantees that $\d_{X\otimes Y}^2=0$.

In fact $\CS(\k)$ has the structure of a symmetric monoidal category, with the braiding morphisms defined by
\[
\tau_{X,Y}:X\otimes_\k Y \rightarrow Y\otimes_\k X,\qquad\qquad x\otimes y \mapsto (-1)^{|x||y|} y\otimes x.
\]
The sign here is necessary in order to guarantee naturality of $\tau_{X,Y}$ with respect to morphisms of nonzero degree:
\[
(g\otimes f)\circ \tau_{X,Y} = \tau_{X',Y'}\circ (f\otimes g)
\]
for all $f\in \Hom_\CS(X,X')$ and all $g\in \Hom_\CS(Y,Y')$.

\subsubsection{DG categories}
\label{sss:dgcats}
\begin{definition}\label{def:dgcat}
A $\k$-linear category $\CS$ is called a \emph{dg category} if each hom space $\Hom_\CS(X,Y)$ is a $\Z$-graded complex of $\k$-modules and composition of morphisms defines a chain map
\[
\Hom_\CS(Y,Z)\otimes_\k \Hom_\CS(X,Y)\rightarrow \Hom_\CS(X,Z).
\]
Equivalently, a dg category is a category enriched in $\CS(\k)$.
\end{definition}
The differential in the hom space $\Hom_\CS(X,Y)$ is denoted simply by $d$ or $d_\CS$.  Since the tensor product differential on $\Hom_\CS(Y,Z)\otimes_\k \Hom_\CS(X,Y)$ is given by $d_\CS\otimes \Id + \Id\otimes d_\CS$, the condition that composition of morphisms be a chain map is equivalent to the Leibniz rule
\[
d_\CS(f\circ g) = d_\CS(f)\circ g + (-1)^{|f|}f\circ d_\CS(g),
\]
where $|f|\in\Z$ is the degree of $f$.

\begin{example}
If $\AS$ is a $\k$-linear category, then we let $\Ch(\AS)$ denote the category of complexes over $\AS$, with hom spaces given by the complexes $\Hom_{\Ch(\AS)}(X,Y)$ with
\[
\Hom_{\Ch(\AS)}^k(X,Y) = \prod_{i\in \Z}\Hom_\AS(X^i,Y^{i+k})
\]
and differential given by the super-commutator
\[
f\mapsto [\d,f]:=\d_Y\circ f - (-1)^{\deg(f)} f\circ \d_X .
\]
It is an exercise to check that $[\d,[\d,f]]=0$, and the Leibniz rule is satisfied, so that $\Ch(\AS)$ is a dg category.
\end{example}

The following is standard terminology.
\begin{definition}\label{def:closed etc}
A morphism $f\in \Hom_\CS(X,Y)$ is \emph{closed} if $d_\CS(f)=0$ and \emph{exact} if $f=d_\CS(h)$ for some $h\in \Hom_\CS(X,Y)$.  Closed (resp.~exact) morphisms are called \emph{cycles} (resp.~\emph{boundaries}).  Exact morphisms are also called \emph{null-homotopic}.  More generally, two morphisms $f,g\in \Hom_\CS^l(X,Y)$ are \emph{homotopic} (written $f\simeq g$) if $f-g$ is exact.

We let $Z^0(\CS)$ denote the category with the same objects as $\CS$ but morphisms the degree zero closed morphisms in $\CS$.

The \emph{homotopy category} (or \emph{cohomology category}) $H^0(\CS)$ is the category with the same objects, but morphisms the degree zero cycles mod boundaries.

An \emph{isomorphism} in $\CS$ is a closed degree zero invertible morphism in $\CS$.   A closed degree zero morphism $f\in \Hom_\CS(X,Y)$ is a \emph{homotopy equivalence} if it represents an isomorphism in $H^0(\CS)$.  We say that a collection of maps
\[
\begin{tikzpicture}[baseline=0em]
\tikzstyle{every node}=[font=\small]
\node (a) at (0,0) {$X$};
\node (b) at (3,0) {$Y$};
\path[->,>=stealth,shorten >=1pt,auto,node distance=1.8cm, thick]
(a) edge [loop left, looseness = 5,in=215,out=145] node[left] {$h$}	(a)
([yshift=3pt] a.east) edge node[above] {$f$}		([yshift=3pt] b.west)
([yshift=-2pt] b.west) edge node[below] {$g$}		([yshift=-2pt] a.east)
(b) edge [loop left, looseness = 5,in=35,out=325] node[right] {$k$}	(b);
\end{tikzpicture}
\]
are the \emph{data of a homotopy equivalence} $X\simeq Y$ if $\deg(f)=\deg(g)=0$, $\deg(h)=\deg(k)=-1$, and
\[
d_\CS(f) = d_\CS(g) =0,\qquad\qquad d_\CS(h) = \Id_X - g\circ f,\qquad \qquad d_\CS(k) = \Id_Y-f\circ g.
\]
An object $X\in \CS$ is \emph{contractible} if $X\simeq 0$ or, equivalently, $\Id_X$ is an exact morphism.  In this case any $h\in \End^{-1}_\CS(X)$ with $d_\CS(h) = \Id_X$ is called a \emph{null-homotopy} or \emph{contracting homotopy} for $X$.
\end{definition}

\begin{remark}
We emphasize that there may be many invertible morphisms in $\CS$, but the word \emph{isomorphism} is reserved for the closed degree zero invertible morphisms.  We use $\cong$ to denote isomorphism in $\CS$ (equivalently, isomorphism in $Z^0(\CS)$) and $\simeq$ to denote isomorphism in $H^0(\CS)$.
\end{remark}
\begin{remark}
An ordinary $\k$-linear category $\AS$ can be regarded as a dg category with $Z^0(\AS)=\AS$ (hence also $H^0(\AS)=\AS$).
\end{remark}

\begin{remark}[Consequences of the Leibniz rule]\label{rmk:consequences}
Note that identity morphisms are automatically closed by applying the Leibniz rule to $d_{\CS}(\Id_X) = \d_{\CS}(\Id_X\circ \Id_X)$.

If $f$ is an invertible (not necessarily degree zero) morphism $\CS$ then applying the Leibniz rule to $f\circ f\inv$ yields
\begin{equation}\label{eq:d of inverse}
d_\CS (f\inv) \ = \ (-1)^{|f|+1} f\inv\circ d_\CS(f)\circ f\inv.
\end{equation}
In particular the inverse of a closed morphism is automatically closed.
\end{remark}

\begin{example} If $\AS$ is a $\k$-linear category the degree zero cycles and boundaries in the dg category of complexes $\Ch(\AS)$ are exactly the degree zero chain maps and null-homotopic chain maps, respectively. Thus the cohomology category $H^0(\Ch(\AS))=:\KC(\AS)$ is the usual homotopy category of complexes.
\end{example}

\begin{remark}\label{rmk:othergradings}
One can consider dg categories which are graded by abelian groups other than $\Z$.  If $\Gamma$ is an abelian group then in order to develop the theory of differential $\Gamma$-graded categories we need to endow $\Gamma$ with two additional pieces of data:
\begin{enumerate}
\item a symmetric bilinear form $\ip{\ ,\ }:\Gamma\rightarrow \Z/2$.
\item a distinguished element $\iota\in \Gamma$ with $\ip{\iota,\iota}=1$.
\end{enumerate}
The distinguished element $\iota$ is the degree of all differentials, and the form $\ip{\ ,\ }$  is used in the formulation of the Koszul sign rule for ``commuting things past one another''.  More precisely, it is used in defining the symmetric monoidal structure in the category of $\Gamma$-graded $\k$-modules.  This is necessary to define the notions of natural transformation and opposite category, among other things.  Moreover, the group homomorphism $p:\Gamma\rightarrow \Z/2$ sending $j\mapsto \ip{j,\iota}$ enters in the definition of the Leibniz rule.

For simplicity we often stick to the case of $\Z$- or $\Z/2$-graded dg categories with $\ip{i,j}:=ij$ (mod 2) and distinguished element $\iota=1$.

\end{remark}

\subsubsection{DG functors and natural transformations}
\label{sss:functors and natural trans}
If $\CS$ and $\DS$ are dg categories, a dg functor $F:\CS\rightarrow \DS$ is a functor whose action on hom spaces is a degree zero chain map
\[
\Hom_\CS(X,Y)\rightarrow \Hom_{\DS}(F(X),F(Y)),\qquad\qquad f\mapsto F(f).
\]
In other words, $d_\DS(F(f)) = F(d_\CS(f))$.  Given dg functors $F,G:\CS\rightarrow \DS$, a natural transformation $\phi:F\rightarrow G$ of degree $k$ is a collection of morphisms $\phi_X\in \Hom_\DS^k(F(X),G(X))$ such that all of the relevant squares involving morphisms $F(f)$, $G(f)$ super-commute.  That is to say,
\[
\phi_Y\circ F(f) = (-1)^{k|f|} G(f)\circ \phi_X
\]
for all morphisms $f\in \Hom_\CS(X,Y)$.  The differential $d(\phi)$ of such a natural transformation is defined by $d(\phi)_X = d(\phi_X)$.  To check naturality of $d(\phi)$ is an easy exercise (hint: take differential of the equation $0 = \phi_Y\circ F(f) - (-1)^{|\phi||f|}G(f)\circ \phi_X$).

\subsection{Direct sum and suspension}
\label{ss:suspension}


\begin{definition}[Additive and suspended envelope]\label{def:add sus hull}
If $\CS$ is a dg category, let $\Sigma \CS$ denote the dg category whose objects are formal expressions $\bigoplus_{i\in I} X_i[b_i]$ in which $I$ is a finite set, $X_i\in \CS$, and $b_i\in \Z$.   The morphism complexes in $\Sigma \CS$ are defined by
\[
\Hom_{\Sigma \CS}^l\Big(\bigoplus_{i\in I} X_i[a_i],\bigoplus_{i\in J} Y_j[b_j]\Big) \ := \ \bigoplus_{(j,i)\in J\times I} \Hom_{\CS}^{l+b_j - a_i}(X_i,Y_j),
\]
with differential
\[
d_{\Sigma \CS}((f_{ji})) = ((-1)^{b_j} d_{\CS}(f_{ji}))
\]
and composition
\[
(f_{kj})\circ (g_{ji}) := \left(\sum_j f_{kj}\circ g_{ji}\right).
\]
\end{definition}

If $I=\{1,\ldots,r\}$ and $J=\{1,\ldots,s\}$ then the components $f_{ji}\in  \Hom_\CS(X_i,Y_j) $ of such a morphism can be arranged in to an $s\times r$ matrix
\[
\sqmatrix{
f_{11} & \cdots & f_{1r} \\
\vdots & \ddots & \vdots \\
f_{s1} & \cdots & f_{sr}
}
\ : \ \bigoplus_{i=1}^r X_i[a_i]\rightarrow \bigoplus_{j=1}^s Y_j[b_j].
\]
Then composition is just matrix multiplication together with composition in $\CS$, and the differential just takes the differential component-wise up to a sign (depending only on the row):
\[
d_{\Sigma\CS}\left(\sqmatrix{
f_{11} & \cdots & f_{1r} \\
\vdots & \ddots & \vdots \\
f_{s1} & \cdots & f_{sr}
}\right)
\ \ = \ \
\sqmatrix{
(-1)^{b_1} d_{\CS}(f_{11}) & \cdots & (-1)^{b_1}d_{\CS}(f_{1r}) \\
\vdots & \ddots & \vdots \\
(-1)^{b_s}d_{\CS}(f_{s1}) & \cdots & (-1)^{b_s}d_{\CS}(f_{sr})
} 
\]

\begin{remark}
Given objects $X,X'\in \CS$ and $k\in \Z$ we write $X'\cong X[k]$ if $X'$ is equipped with a closed invertible morphism $\theta\in \Hom_\CS^k(X',X)$.  To understand and remember the signs in the above definition, suppose $X'\cong X[k]$ and $Y'\cong Y[l]$.  Then we have an identification of hom spaces
\[
\Hom_\CS(X',Y')\rightarrow \Hom_\CS(X,Y),\qquad\qquad f\mapsto \theta\circ f\circ \phi\inv
\]
where $\phi\in \Hom_\CS(X',X)$ and $\theta\in \Hom_\CS(Y',Y)$ are chosen closed invertible maps of degrees $k$ and $l$, respectively.  This identification adjusts gradings according to
\[
\deg(\theta\circ f\circ \phi\inv) = \deg(f) + l-k
\]
and also respects differentials only up to a sign:
\[
d_\CS(\theta\circ f\circ \phi\inv) = (-1)^l \theta \circ d_\CS(f)\circ \phi\inv
\]
by the Leibniz rule (since $\theta$ and $\phi\inv$ are closed).
\end{remark}
\begin{remark}
Whether or not $X[k]$ exists as an object of $\CS$ is a property, not a structure, since any two objects $Z_1,Z_2\in \CS$ equipped with closed invertible degree $k$ morphisms $\theta_i:Z_i\rightarrow X$ are canonically isomorphic via $\theta_2\inv\circ \theta_1$.  We say that $\CS$ is \emph{suspended} or \emph{closed under the suspension} if $X[k]$ exists as an object of $\CS$ for all $X\in \CS$ and all $k\in \Z$ (equivalently $X[\pm 1]$ exist in $\CS$ for all $X\in \CS$).  In this case the canonical dg functor $\CS\rightarrow \Sigma\CS$ is essentially surjective as well as fully faithful, hence is an equivalence of dg categories.

The same argument shows that $\Sigma\Sigma\CS\simeq \Sigma \CS$ always.
\end{remark}

\begin{example}\label{ex:susp of cxs}
If $\CS=\Ch(\AS)$ and $X\in \Ch(\AS)$ is a complex, then let $X[1]$ denote the complex with chain objects $X[1]^k = X^{k+1}$ (so $X[1]$ is $X$ shifted to the left) and differential $d_{X[1]} = -d_X$. 

Observe that the identity $\Id_X$ can be regarded as a closed degree 1 closed morphism $\theta_X:X[1]\rightarrow X$ (for $\theta_X$ to be closed it is important that the differential on $X[1]$ is negated).  Thus the notation for $X[1]$ is the suspension of $X$ in the sense described above, thereby justifying the notation.  If $\AS$ is additive then it follows that $\Ch(\AS)$ is also additive as well as suspended, and $\Sigma\Ch(\AS)\simeq \Ch(\AS)$.
%
\end{example}

\begin{example}
Let $\AS$ be an additive $\k$-linear category, which we may regard as a dg category with trivial grading and differential.   It is an easy exercise in universal properties to show that $(X\oplus Y)[k]\cong X[k]\oplus Y[k]$, where the former `$\oplus$' denotes direct sum in $\AS$ and the latter denotes formal direct sum, in $\Sigma\AS$.  Thus each object in $\Sigma\AS$ is isomorphic (after collecting like terms) to a formal direct sum of the form $\bigoplus_{k\in \Z} X^k[-k]$ with $X^k=0$ for all but finitely many $k$. Such objects may be visualized as sequences, or $\Z$-graded objects of $\CS$, in which $X^k$ is placed in degree $k$ (using the cohomological convention for shift $[1]$).

Thus, $\Sigma \AS$ is equivalent to the dg category of finite complexes over $\AS$ with zero differential.  This dg category has trivial differential but nontrivial $\Z$-grading.  The process of ``turning on differentials'' is sometimes called twisting, and can be expressed in the language of abstract dg categories.  We consider this next.
\end{example}
\subsection{Twists}
\label{ss:twists}

\begin{definition}
If $\CS$ is a dg category, then let $\Tw(\CS)$ denote the category whose objects are formal expressions $\tw_\a(X)$ with $X\in \CS$ and $\a\in \End_\CS(X)$ is a degree 1 endomorphism satisfying the \emph{Maurer-Cartan} equation
\begin{equation}\label{eq:MC}
d_\CS(\a)+\a\circ \a = 0.
\end{equation}
Morphism spaces in $\Tw(\CS)$ are defined by
\[
\Hom_{\Tw(\CS)}^l\Big(\tw_\a(X),\tw_\b(Y))\Big) \ := \ \Hom_\CS^l(X,Y)
\]
with twisted differential
\[
d_{\Tw(\CS)}(f) = d_{\CS}(f) +\b\circ f - (-1)^{|f|}f\circ \a.
\]
The Leibniz rule for $d_{\Tw(\CS)}$ is easily checked, as is the fact that $d_{\Tw(\CS)}^2=0$.  Thus, $\Tw(\CS)$ defines a dg category.  This category is called the \emph{twisted envelope of $\CS$}.  The category $\Tw(\Sigma\CS)$ (see Definition \ref{def:add sus hull}) is called the category of \emph{twisted complexes} over $\CS$ \cite{BonKap90}.
\end{definition}

To  motivate the definition of $\Tw(\CS)$, consider the following.  Let $X,X'\in \CS$ be objects.  We say that $X'$ is a \emph{twist}  of $X$ if there is a degree zero invertible morphism $\phi:X\rightarrow X'$ (not necessarily closed).  Given such a morphism $\phi$, we define an endomorphism $\a\in \End_\CS(X)$ of degree 1 by $\a:=\phi\inv\circ d_\CS(\phi)$ or, equivalently, either of the equations
\[
d_\CS(\phi) = \phi\circ \a,\qquad\qquad d_\CS(\phi\inv)=-\a\circ \phi\inv.
\]
The fact that $\a$ satisfies equation \eqref{eq:MC} is an easy computation:
\[
0= d_\CS^2(\phi) = d_\CS(\phi\circ \a) = d_{\CS}(\phi)\circ \a + \phi\circ d_{\CS}(\a) = (\phi\circ \a)\circ \a + \phi\circ d_\CS(\a).
\]
In this case we say that $X'$ is a (the) \emph{twist of $X$ by $\a$} and we write $X'\cong \tw_\a(X)$. 

Then $\Tw(\CS)$ can be thought of as the category obtained from $\CS$ by formally adjoining all twists of all objects.  The nature of hom complexes $\Tw(\CS)$ can be explained as follows.    If $X'\cong \tw_\a(X)$ and $Y'\cong \tw_\b(Y)$ then we have an identification of hom spaces
\[
\Hom_{\CS}(X',Y')\rightarrow \Hom_\CS(X,Y),\qquad\qquad g\mapsto \psi\inv \circ g \circ \phi
\]
where $\phi:X\rightarrow X'$ and $\psi:Y\rightarrow Y'$ are degree zero invertible maps with $d_\CS(\phi) = \phi\circ \a$ and $d_\CS(\psi)=\psi\circ \b$.  This identification preserves grading but not differentials.  Indeed, the differential on $\Hom_\CS(X,Y)$ induced by that on $\Hom_\CS(X',Y')$ is given by
\begin{eqnarray*}
f &\mapsto & \psi\inv\circ d_\CS(\psi\circ f\circ \phi\inv )\circ \phi\\
&= & \psi\inv\circ \Big(d_\CS(\psi)\circ f\circ \phi\inv +\psi\circ d_\CS(f)\circ \phi\inv + \psi\circ f\circ d_\CS(\phi\inv)\Big)\circ \phi\\
&= & \psi\inv\circ \Big(\psi\circ \b\circ f\circ \phi\inv +\psi\circ d_\CS(f)\circ \phi\inv - \psi\circ f\circ \a\circ \phi\inv\Big)\circ \phi\\
&= &  \b\circ f\ +  d_\CS(f)- f\circ \a.
\end{eqnarray*}
This explains the definition of $\Tw(\CS)$.

\begin{remark}
We say that $\CS$ is \emph{closed under twists} if the twist of $X$ by $\a$ exists as an object of $\CS$ for all $X\in \CS$ and all Maurer-Cartan endomorphisms $\a\in \End_\CS(X)$.  In this case the canonical dg functor $\CS\rightarrow \Tw(\CS)$ is an equivalence of dg categories.  It is not hard to see that $\Tw(\CS)$ is closed under twists, so $\Tw(\Tw(\CS))\simeq \Tw(\CS)$ always (see \S \ref{sss:additivity} below).
\end{remark}

\begin{example}
Let $\AS$ be a $\k$-linear category and, and let $\AS^\Z$ be the dg category category whose objects are sequences of objects in $\AS$, regarded as a dg category with zero differential (but nontrivial grading on hom spaces).  Then $\Tw(\AS^\Z) =\Ch(\AS)$, the usual dg category of complexes.
\end{example}

\begin{example}
In the case $\CS=\Ch(\AS)$ and $(X,\d_1)$, $(X,\d_2)$ are complexes, the identity map gives a degree zero invertible morphism $(X,\d_1)\rightarrow (X,\d_2)$ whose associated Maurer-Cartan endomorphism is $\d_{2}-\d_1$. 

Conversely, $\a$ is a Maurer-Cartan endomorphism of $(X,\d)$ if and only if $[\d,\a]+\a^2=0$, which occurs if and only if $(\d+\a)^2=0$.  In this case $(X,\d+\a)$ is the twist of $(X,\d)$ by $\a$.

Finally, $\Tw(\Ch(\AS))$ is equivalent to $\Ch(\AS)$ via the dg functor $\Tw(\Ch(\AS))\rightarrow \Ch(\AS)$ sending $\tw_\a((X,\d))\mapsto (X,\d+\a)$.

\end{example}

\subsubsection{Additivity of twists}
\label{sss:additivity}

Let $\a,\b\in \End^1_\CS(X)$ be endomorphisms such that $\a$ and $\a+\b$ both satisfy the Maurer-Cartan equation.  Consider $\id_X$ as a degree zero invertible morphism $\phi:\tw_\a(X)\rightarrow \tw_{\a+\b}(X)$ in $\Tw(\CS)$.  The differential of this morphism is
\[
d_{\Tw(\CS)}(\phi) = 0 + (\a+\b)\circ \id_X - \id_X\circ \a = \b = \phi\circ \b,
\]
which shows that
\begin{equation}\label{eq:tw additivity}
\tw_{\a+\b}(X)\cong \tw_\b(\tw_\a(X)).
\end{equation}
Conversely, $\b\in \End^1_\CS(X)$ represents a Maurer-Cartan endomorphism of $\tw_\a(X)$ if and only if
\[
d_{\CS}(\b) + \a\circ \b + \b\circ \a +\b^2 =0
\]
which, in the presence of $d_\CS(\a)+\a^2=0$, is equivalent to
\[
d_\CS(\a+\b) + (\a+\b)^2=0.
\]

Thus, $\tw_\b(\tw_\a(X))$ is a well-defined object of $\Tw(\Tw(\CS))$ if and only if $\tw_\a(X)$ and $\tw_{\a+\b}(X)$ are well-defined objects of $\Tw(\CS)$, in which case we have a canonical isomorphism \eqref{eq:tw additivity}.

\subsection{Mapping cones}
\label{ss:cones}
\begin{definition}\label{def:cone}
If $f:X\rightarrow Y$ is a closed degree zero morphism in $\CS$ then $\Cone(f)$ is defined to be the twisted complex
\[
\Cone(f) \ = \ \tw_{\gamma}(X[1]\oplus Y) \ \in \ \Tw(\Sigma\CS),\qquad\quad \gamma = \sqmatrix{0&0\\f&0}.
\]
More generally, if $f:\tw_\a(X)\rightarrow \tw_\b(Y)$ is a closed morphism in $\Tw(\Sigma\CS)$ then
\[
\Cone(f) \ = \ \tw_{\gamma}(X[1]\oplus Y) \ \in \ \Tw(\Sigma\CS),\qquad\quad \gamma = \sqmatrix{-\a&0\\f&\b},
\]
which corresponds to the image of
\[
\tw_{\smMatrix{0&0\\f&0}}\left(\tw_{\smMatrix{-\a&0\\0&\b}}(X[1]\oplus Y)\right)
\]
under the equivalence $\Tw(\Tw(\Sigma \CS))\rightarrow \Tw(\Sigma \CS)$.
\end{definition}

The following is standard (see Theorem 2.10 in Chapter 4 of Spanier \cite{spanier}).
\begin{lemma}\label{lemma:cones detect equiv}
Let $f\in \Hom^0_\CS(X,Y)$ be closed.  Then $\Cone(f)$ is contractible if and only if $f$ is a homotopy equivalence.
\end{lemma}
\begin{proof}
Let $H\in \End^{-1}_\CS(X[1]\oplus Y)$ be given, written in terms of components as
\[
H : = \sqmatrix{-h & g \\ m & k}.
\]
The differential of $H$, regarded as an endomorphism of $\Cone(f) = \tw_{\smMatrix{0&0\\f&0}}(X[1]\oplus Y)$ is
\begin{eqnarray*}
d_{\End(C_f)}(H)
& =&\sqmatrix{d(h) & -d(g) \\ d(m) & d(k)} +   \sqmatrix{0 & 0 \\ f & 0}\sqmatrix{-h & g \\ m & k} + \sqmatrix{-h & g \\ m & k}
\sqmatrix{0 & 0 \\ f & 0}\\
& =&\sqmatrix{d(h) & -d(g) \\ d(m) & d(k)} +   \sqmatrix{g  f & 0 \\ -f  h + k f & f  g}.
\end{eqnarray*}
This equals $\Id_{\Cone(f)}$ if and only if
\begin{enumerate}
\item $d(h) = \Id_X - g  f$.
\item $d(g) = 0$.
\item $d(m) = fh -  kf$.
\item $d(k) = \Id_Y - f  g$.
\end{enumerate}
Thus, if $\Cone(f)$ is contractible then there is a closed $g\in \Hom(Y,X)$ such that $f g\simeq \Id_Y$ and $g f\simeq \Id_X$.

To complete the proof, we have to show that if we are given morphisms $g\in \Hom^{-1}_\CS(Y,X)$, $h\in \End^{-1}_\CS(X)$, $k\in \End^{-1}_\CS(Y)$ satisfying (1), (2), (4) above, then these data can always be chosen so that the degree 0 cycle $z:=kf- fh$  is a boundary (this is exactly the content of equation (3) above).

To this end, assume that are given $f,g,h,k$ satisfying $d(f)=0$ as well as (1), (2), (4) above. Define the following elements
\[
z:=fh-kf,\qquad h':=h-gz, \qquad  m:= kz.
\] 

An easy exercise shows that
\begin{eqnarray*}
d(h') & =& \Id_X - gf\\
d(m) & =& fh' - kf
\end{eqnarray*}
Thus $(-h',g,m,k)$ are the components of a null-homotopy for $\Cone(f)$.
\end{proof}

\subsection{Twists of contractible objects}
\label{ss:twists of contractible}
The following question is quite natural, and initiates the subject of homological perturbation theory.

\begin{question}
If $X\in \CS$ is contractible and $\tw_\a(X)$ is a twist of $X$, then under what additional conditions is $\tw_\a(X)$ contractible?
\end{question}
Clearly some conditions are necessary.  For instance if $\CS= \Ch(\AS)$ is a dg category of complexes and $(X,\d_X)$ is a contractible complex then $\tw_{-\d_X}(X,\d_X) = (X,0)$  is not contractible except in the trivial case $X=0$.  The following is well known.

\begin{lemma}[Twists of contractible complexes]\label{lemma:tw of zero}
Let $X\in \CS$ be an object and $h_0\in \End^{-1}_\CS(X)$ such that $d_\CS(h_0)=\Id_X$.  If $\tw_\a(X)$ is a twist of $X$ such that $\Id_X+\a\circ h_0$ is invertible in $\End^0_\CS(X)$ then $\tw_\a(X)$ is contractible, with null-homotopy
\[
h \ := \ h_0\circ (\Id_X+\a\circ h_0)\inv.
\]
\end{lemma}
Note that $(\Id_X+h_0\circ \a)\inv = \Id_X - h_0\circ (\Id_X+\a\circ h_0)\inv \circ \a$, and the above null-homotopy can also be written as
\[
h\ := \ (\Id_X+h_0\circ \a)\inv\circ h_0.
\]
\begin{proof}
Exercise. It is useful to keep in mind the rule
\[
d_\CS(\phi\inv) = - \phi\inv\circ d_\CS(\phi)\circ \phi\inv.
\]
for all invertible morphisms $\phi\in \End_\CS^0(X)$.
\end{proof}

\begin{remark}\label{rmk:loc finiteness}
In most applications of Lemma \ref{lemma:tw of zero} the null-homotopy $h$ is well-defined because the series expansion
\[
(\Id_X+\a\circ h_0)\inv \ \ = \ \ \sum_{k\geq 0}(-1)^k h_0\circ (\a\circ h_0)^k
\]
satisfies appropriate local finiteness conditions.  For instance if $X$ admits a finite direct sum decomposition $X =\bigoplus_{i\in I} X_i$ with respect to which
\begin{enumerate}
\item the initial null-homotopy $h_0$ is represented by a diagonal (or lower triangular) matrix,
\item the twist $\a$ is represented by a strictly lower triangular matrix with respect to some partial order on $I$,
\end{enumerate}
then $\Id_X+\a\circ h_0$ is represented by a lower triangular matrix with $\Id_{X_i}$ on the diagonal, hence is invertible. 
\end{remark}

\begin{remark}
In the setup of Remark \ref{rmk:loc finiteness} the hypothesis that $I$ be finite can be replaced with the hypothesis that $I$ satisfy the \emph{upper finiteness} or \emph{ascending chain} condition.  See Corollary \ref{cor:onesided replacement} (with curvature $z=0$).
\end{remark}

\section{Strong homotopy equivalences and curvature}
\label{s:she curved}
Let $\CS$ be a dg category.  Above we saw how to construct dg categories $\Sigma\CS$ and $\Tw(\CS)$.  The composite $\Tw(\Sigma \CS)$ is called the category of twisted complexes over $\CS$; if $\AS$ is a $\k$-linear additive category then $\Tw(\Sigma \AS)$ is just the usual category of bounded complexes over $\AS$.   In this section we consider a generalization of this construction to the category of \emph{curved twisted complexes}.

Even though the notion of curved complexes may initially seem exotic, many important concepts in dg categories (for instance strong homotopy equivalence, as we will see) are most naturally expressed in the language of curved complexes. 

\subsection{Curved twisted complexes}
\label{ss:curved tw}
\begin{definition}
Let $\CS$ be a dg category.  A \emph{curved twisted complex} in $\CS$ is a triple $(X,\a,z)$ where $X\in \CS$, $z\in \End_\CS(X)$ is a degree 2 endomorphism (not necessarily closed), and $\a\in \End_\CS(X)$ is a degree 1 endomorphism satisfying the Maurer-Cartan equation (with curvature $z$)
\[
d(\a)+\a^2 = z.
\]
A \emph{morphism} of curved twisted complexes $(X,\a,z_X)\rightarrow (Y,\b,z_Y)$ is a morphism $f:X\rightarrow Y$ such that $z_Y\circ f - f\circ z_X$. The differential of $f$, regarded as a morphism of curved twisted complexes, is
\[
d_{\twc(\CS)}(f) = d_\CS(f) + \b\circ f - (-1)^{|f|}f\circ \a.
\]
With these definitions, the curved twisted complexes form a dg category, denoted $\cTw(\CS)$.
\end{definition}

\begin{example}
If $\CS =\AS^\Z$ is the category of $\Z$-graded objects, then a curved twisted complex over $\CS$ is just a triple $(X,\d,z)$ where $X\in \AS^\Z$ and $\d$ is a degree 1 endomorphism of $X$ satisfying $\d^2=z$.  Thus a curved twisted complex in this case is just a ``complex'' whose differential does not square to zero.  We will refer to such objects as curved complexes.

Morphisms of curved complexes need not commute with the differentials $\d$, but they do need to commute with their curvatures $\d^2$.
\end{example}

\begin{example}
Let $\e$ be a formal indeterminate of degree 2 satisfying $\e^2=0$, and let $\CS=\AS^{\Z}\otimes \k[\e]/\e^2$.  Objects of this category are the same as objects of $\AS$, but we adjoin $\e$ to all hom complexes.  In other words a degree $l$ morphism $f:X\rightarrow Y$ in $\CS$ is a formal sum
\[
f_0+ \e f_1
\]
with $f_0\in \Hom_{\AS^\Z}^l(X,Y)$, and $f_1\in \Hom_{\AS^\Z}^{l-2}(X,Y)$, and composition is defined by
\[
(f_0+ \e f_1)\circ (g_0+\e g_1)  = f_0\circ g_0 + \e(f_1\circ g_0+f_0\circ g_1).
\]
Let $(X,\d)\in \Ch(\AS)$ be an ordinary chain complex and $t:X\rightarrow X$ a degree zero chain map which is null-homotopic: $t = \d\circ h + h\circ \d$ for some degree $-1$ map $h:X\rightarrow X$.  Then $\d+\e h$ gives $X$ the structure of a curved complex with curvature $\e t$, since
\[
(\d+\e h) = \d^2 + \e (\d\circ h + h\circ \d) = 0+\e t.
\]

Thus $\cTw(\CS)$ contains as a full dg subcategory a category whose objects are complexes $(X,\d)$ equipped with a null-homotopic endomorphism (and a choice of trivializing homotopy $h$).

As a special case we find that the data of a curved complex $(X, \d+\e h, \e \Id)$ is equivalent to the data of a contracting homotopy for the complex $(X,\d)$.
\end{example}

\subsection{Strong homotopy equivalence}
\label{ss:she}

We introduce the notion of strong homotopy equivalence between two objects of a dg category and reinterpret in the language of curved twisted complexes.  The definition below is adapted from \cite{MarklIdeal}.

\begin{definition}\label{def:Rshe}
Let $\e$ be a formal indeterminate of cohomological degree 2.  Let $\she$ denote the $\k\llbracket \e \rrbracket$-linear dg category with two objects $X,Y$ and morphisms freely generated by morphisms $f,g,h,k$ where $f:X\rightarrow Y$ and $g:Y\rightarrow X$ have degree zero, $h:X\rightarrow X$ and $k:Y\rightarrow Y$ have degree $-1$, with the differentials defined by
\begin{subequations}
\begin{eqnarray}
d(f)&=& \e (fh -kf)\label{eq:she df}\\
d(g) &=& \e(gk- hg)\label{eq:she dg}\\
d(h)&=&\Id_X-gf - \e h^2\label{eq:she dh}\\
d(k)&=&  \Id_Y-fg - \e k^2\label{eq:she dk}
\end{eqnarray}
\end{subequations}
\end{definition}
Since we impose no relations between morphisms, the differential extends uniquely to arbitrary morphisms in $\she$ via the Leibniz rule.  However, it must be checked that $d^2$ is zero on morphisms.  It suffices to check this on the generating morphisms, which is an easy exercise.  For instance, to check that $d^2(f)=0$ we must check that that right hand side of \eqref{eq:she df} is automatically closed.  Indeed:
\[
d(fh)=d(f)h + fd(h) = \e(fh-kf)h +f(1-gf-\e h^2) = f-fgf - \e  kfh,
\]
and
\[
d(kf) =d(k) f - kd(f) = (1-fg-\e k^2)f - \e k(fh-kf) = f-fgf - \e  kfh
\]
and it follows that $d^2(f)=0$.  The other identities $d^2(g)=d^2(h)=d^2(k)=0$ are proven just as easily.

Now, let $\CS$ be a dg category and consider the dg category $\CS\llbracket \e\rrbracket$.  Objects of this category are the same as objects of $\CS$, and morphism complexes are obtained from those in $\CS$ by extending scalars from $\k$ to $\k\llbracket \e \rrbracket$.

\begin{definition}\label{def:she}
A \emph{strong homotopy equivalence} in a dg category $\CS$ is a $\k\llbracket\e\rrbracket$-linear dg functor $F:\she\rightarrow \CS\llbracket \e \rrbracket$.  More concretely, a strong homotopy equivalence consists of a pair of objects $X,Y\in \CS$ and formal series of morphisms
\[
f(\e) = \sum_{i\geq 0} f_i \e^i\qquad (f_i\in \Hom^{-2i}_\CS(X,Y)),
\]
\[
g(\e) = \sum_{i\geq 0} g_i \e^i\qquad (g_i\in \Hom^{-2i}_\CS(Y,X))
\]
\[
h(\e) = \sum_{i\geq 0} h_i \e^i \qquad (h_i\in \Hom^{-1-2i}_\CS(X,X))
\]
\[
k(\e) = \sum_{i\geq 0} k_i \e^i \qquad (k_i\in \Hom^{-1-2i}_\CS(Y,Y))
\]
satisfying relations \eqref{eq:she df}, \eqref{eq:she dg}, \eqref{eq:she dh}, and \eqref{eq:she dk}.  Here we are abusing notation, denoting objects and morphisms in $\she$ and their images in $\CS\llbracket\e \rrbracket$ by the same notation.  
\end{definition}
Setting $\e=0$, we see in particular that $f_0,g_0,h_0,k_0$ give the data of a homotopy equivalence $X\simeq Y$ in $\CS$.  Let us attempt to explain the (obstruction theoretic) meaning of some of the other components $f_i,g_i,h_i,k_i$.  The degree $\e^1$ part of \eqref{eq:she df}, for instance is equivalent to
\[
d(f_1) = f_0h_0 - k_0f_0.
\]
In other words, the homotopy equivalence $f_0$ intertwines the homotopies $h_0,k_0$ up to homotopy.  The component $g_1$ admits a similar interpretation.

The degree $\e^1$ part of \eqref{eq:she dh} yields
\[
d(h_1) = -g_1f_0 + g_0 f_1 - h_0^2.
\]
The right hand side is some degree $-2$ element of $\End_\CS(X)$ (automatically closed from the proof that $d_{s.h.e.}^2=0$), which is null-homotopic via $h_1$.  In this way all the components $f_i,g_i,h_i,k_i$ for $i\geq 1$ can be realized as killing an infinite family of obstructions constructed from themselves (and $f_0,g_0,h_0,k_0$).

\begin{remark}
Markl \cite{MarklIdeal} has shown that the dg category $\she$ which represents strong homotopy equivalences is a cofibrant resolution of the dg category $\RS_{\text{iso}}$ which represents isomorphisms (consists of two objects $X$ with closed degree zero mutually inverse isomorphisms $X\leftrightarrow Y$).  
\end{remark}

\subsubsection{Strong homotopy equivalences and curvature}
\label{sss:she and curvature}

If $f,g,h,k$ give a strong homotopy equivalence $X\simeq Y$ in $\CS$, then \eqref{eq:she dh}, \eqref{eq:she dk} are equivalent to the equations
\[
d(\e h) + (\e h)^2 = \e(\Id_X - g\circ f) ,\qquad \qquad d(\e k) + (\e k)^2 = \e(\Id_Y - f\circ g),
\]
which is equivalent to saying $(X, \e h)$ and $(Y,\e k)$ are curved twisted complexes with curvature $\e(\Id_X-g\circ f)$ and $\e(\Id_Y-f\circ g)$, respectively. 

Moreover, \eqref{eq:she df}, \eqref{eq:she dg} are equivalent to the equations
\[
d(f) + (\e k) \circ f - f\circ (\e h) = 0,\qquad\qquad d(g) + (\e h)\circ g - g\circ (\e k).
\]
so $f$ and $g$ are closed morphisms of twisted complexes (note also that $f$ and $g$ intertwine the curvatures as they should: $f\circ \e(\Id_X - g\circ f) = \e(\Id_Y-f\circ g)\circ f$).  We collect all of this in the following.

\begin{observation} The data of a strong homotopy equivalence is equivalent to the data of:
\begin{enumerate}
\item curved twisted complexes $(X,\e h, z_X)$, and $(Y,\e k, z_Y)$ in $\cTw(\CS\llbracket\e \rrbracket)$.
\item closed degree zero morphisms $f:(X,\e h, z_X)\leftrightarrow (Y,\e k, z_Y):g$ such that
\[
z_X = \e (\Id_X-g\circ f),\qquad\qquad z_Y= \e(\Id_Y-f\circ g).
\]
\end{enumerate}
\end{observation}

\subsubsection{From homotopy equivalences to strong homotopy equivalences}
\label{sss:he to she}
In this section we show that each $X$ and $Y$ are homotopy equivalent in $\CS$ if and only if they are strongly homotopy equivalent.  We prove this by showing that both are equivalent to the strong contractibility of $\Cone(X\rightarrow Y)$, as defined next.

A \emph{strong null-homotopy} or \emph{strong contracting homotopy} for  $X$ is the data of a strong homotopy equivalence $X\simeq 0$, i.e.~a formal series
\[
h(\e) = \sum_{i\geq 0} h_i \e^i \qquad (h_i\in \Hom^{-1-2i}_\CS(X,X))
\]
satisfying $d(h(\e))+\e h(\e)^2 = \Id_X$.  In this case we call $h(\e)$ a \emph{strong contracting homotopy}.

\begin{lemma}\label{lemma:null implies strongnull}
If $h_0$ is a contracting homotopy for $X$, then
\[
h(\e) := \sum_{k\geq 0} (-1)^k C_k\e^k (h_0)^{2k+1}
\]
is a strong contracting homotopy for $X$, where $C_0,C_1,\ldots$ are the Catalan numbers.
\end{lemma}
\begin{proof}
First observe that $d(h_0)=\Id_X$ implies:
\[
d(h_0^k) = \begin{cases} 0 & \text{ $k$ even} \\ h_0^{k-1} & \text{ $k$ odd}\end{cases}.
\]
The Catalan numbers satisfy the well-known recursion
\[
C_{n} = \sum_{i+j=n-1} C_i C_j ,
\]
adopting the convention that $C_i=0$ for $i<0$.  Thus,
\begin{eqnarray*}
d(h(\e)) + \e h(\e)^2 
& =& \sum_{k\geq 0} (-\e)^k C_k d(h_0^{2k+1}) + \e \sum_{i+j=k} (-\e)^k C_iC_j h_0^{2i+1}h_0^{2j+1}\\
& =& \sum_{k\geq 0} (-\e)^k C_k h_0^{2k} - \sum_{i+j=k-1\atop k\geq 1} (-\e)^k C_iC_j h_0^{2k}.
\end{eqnarray*}
All terms cancel except for the $k=0$ zero term in the sum on the left.  This term equals $\Id_X$.
\end{proof}

The following helps give some perspective on the definition of strong homotopy equivalence.
\begin{lemma}\label{lemma:she iff cone sc}
Let $f_0:X\rightarrow Y$ be a closed degree zero morphism in $\CS$.  Then a matrix of morphisms
\[
H(\e) = \sqmatrix{-h(\e) & g(\e) \\ c(\e) & k(\e)} \ \in \ \End_{\Sigma\CS\llbracket\e\rrbracket}(X[1]\oplus Y)
\]
defines a strong contracting homotopy for $\Cone(f_0)$ if and only if $f_0+\e c(\e)$, $g(\e)$, $h(\e)$, $k(\e)$ define a strong homotopy equivalence $X\simeq Y$.
\end{lemma}
\begin{proof}
Below we will abbreviate by writing $g = g(\e)$, etc.  Let $\tilde{H}$ denote $H$, regarded as an endomorphism of $\Cone(f_0)$ in $\Tw(\Sigma\CS\llbracket\e\rrbracket)$.  From the definition, $H$ is a strong contracting homotopy if and only if
\[
d(\tilde{H}) + \e \tilde{H}^2 = \Id_{\Cone(f_0)},
\]
i.e.
\[
d(H) + \gamma H+H \gamma +\e H^2 = \Id_{\Cone(f_0)},
\]
where $\gamma = \smMatrix{0&0\\ f_0&0}$.    Because of the signs involved in the differential in $\Sigma\CS\llbracket\e\rrbracket$ (Definition \ref{def:add sus hull}) this is equivalent to
\[
\sqmatrix{d(h) & -d(g) \\ d(c) & d(k)} +  \sqmatrix{0&0\\ f_0 & 0} \sqmatrix{-h & g \\ c & k} +\sqmatrix{-h & g \\ c & k} \sqmatrix{0&0\\ f_0 & 0}  + \e  \sqmatrix{-h & g \\ c & k}^2 = \sqmatrix{\Id_{X} & 0 \\ 0 & \Id_Y}.
\]
which is equivalent to the system of equations
\begin{subequations}
\begin{eqnarray*}
d(h) +gf_0+\e h^2 +\e gc  &=& \Id_X \\
-d(g)-\e hg + \e gk &=& 0\\
d(c) - f_0h+kf_0 -\e c h + \e k c &=& 0\\
d(k) + f_0g+\e cg +\e k^2 &=& \Id_Y
\end{eqnarray*}
\end{subequations}
or, equivalently,
\begin{eqnarray*}
d(h) &=& \Id_X-g(f_0+\e c) - \e h^2\\
d(g)&=& \e (gk- hg)\\
d(f_0+\e c) &=& \e((f_0+\e c)h - k(f_0  +\e c))\\
d(k)&=& \Id_Y - (f_0+\e c) g- \e k^2 .
\end{eqnarray*}
This proves the lemma.
\end{proof}

\begin{proposition}\label{prop:he iff she}
Any homotopy equivalence $f_0,g_0,h_0,g_0$ in a dg category $\CS$ lifts to a strong homotopy equivalence, possibly after modification of $h_0$ or $k_0$.
\end{proposition}
\begin{proof}
Suppose we are given the data of a homotopy equivalence $(f_0,g_0,h_0,k_0)$ relating $X$ and $Y$.  That is to say $f_0\in \Hom^0_\CS(X,Y)$, $g_0\in \Hom_\CS^0(Y,X)$, $h\in \End^{-1}_\CS(X)$, and $k\in \End^{-1}_\CS(Y)$ satisfy
\[
d(f_0) = d(g_0)=0,\qquad\quad d(h_0) = \Id_X - g_0f_0,\qquad\quad d(k_0)= \Id_Y - f_0g_0.
\]
After modifying $h_0$ if necessary (as in the proof of Lemma \ref{lemma:cones detect equiv}) we may also suppose we are given $f_1\in \Hom^{-2}_\CS(X,Y)$ such that
\[
d(f_1) + k_0f_0 - f_0 h_0 = 0,
\]
so
\[
H_0:= \sqmatrix{-h_0 & g_0 \\ f_1 & k_0}
\]
defines a null-homotopy for $\Cone(f)$.  It then follows that
\[
H(\e) := \sum_{l\geq 0} (-\e)^l C_l H_0^{2l}
\]
defines a strong contracting homotopy for $\Cone(f)$ by Lemma \ref{lemma:null implies strongnull}.  The data of the strong homotopy equivalence $X\simeq Y$ can then be extracted from the components of $H(\e)$ by Lemma \ref{lemma:she iff cone sc}.  In particular $f(\e),g(\e),h(\e),k(\e)$ can be written explicitly in terms of the inital data $f_0,g_0,h_0,k_0$, and $f_1$.
\end{proof}

\section{Homological perturbation}
\label{s:curved hpl}



\subsection{The curved ideal homological perturbation lemma}
\label{ss:main result}

To begin with, we introduce an auxilliary notion that will help frame the discussion below.

\begin{definition}\label{def:zhe}
Let $X,Y\in \CS$ be objects of a dg category and let $z$ be a degree 2 closed element of the center of $\CS$ (natural transformation of the identity functor).  Then a $z$-homotopy equivalence $X\simeq Y$ is the data of morphisms in $\CS$:
\[
\begin{tikzpicture}[baseline=0em]
\tikzstyle{every node}=[font=\small]
\node (a) at (0,0) {$X$};
\node (b) at (3,0) {$Y$};
\path[->,>=stealth,shorten >=1pt,auto,node distance=1.8cm, thick]
(a) edge [loop left, looseness = 5,in=215,out=145] node[left] {$h$}	(a)
([yshift=3pt] a.east) edge node[above] {$f$}		([yshift=3pt] b.west)
([yshift=-2pt] b.west) edge node[below] {$g$}		([yshift=-2pt] a.east)
(b) edge [loop left, looseness = 5,in=35,out=325] node[right] {$k$}	(b);
\end{tikzpicture}
\]
such that $\deg(f)=\deg(g)=0$, $\deg(h)=\deg(k)=-1$, satisfying
\begin{subequations}
\begin{eqnarray}
d(f)&=& z (fh -kf)\label{eq:zhe df}\\
d(g) &=& z(gk- hg)\label{eq:zhe dg}\\
d(h)&=&\Id_X-gf - z h^2\label{eq:zhe dh}\\
d(k)&=&  \Id_Y-fg - z k^2\label{eq:zhe dk}
\end{eqnarray}
\end{subequations}
\end{definition}

\begin{remark}
 $0$-homotopy equivalence is just a homotopy equivalence in the usual sense.
\end{remark}
\begin{remark}
A strong homotopy equivalence $X\simeq Y$ is the same thing as an $\e$-homotopy equivalence $X\simeq Y$ in $\CS\llbracket \e\rrbracket$.
\end{remark}
\begin{remark}
If $f(\e),g(\e),h(\e),k(\e)$ are the data of a strong homotopy equivalence $X\simeq Y$ then specializing $\e=z$ yields a $z$-homotopy equivalence $X\simeq Y$ in $\CS$, provided that the series
\[
f(z)= \sum_{l\geq 0} f_l z^l,\qquad g(z) = \sum_{l\geq 0} g_l z^l,\qquad h(z)  = \sum_{l\geq 0} h_l z^l,\qquad k(z)= \sum_{l\geq 0} k_l z^l
\]
are well-defined morphisms in $\CS$ (for instance if $z$ is nilpotent or more generally if there is a two-sided ideal $\ng$, with respect to which $\CS$ is complete, containing $z^l$ for some $l\in \N$).
\end{remark}

\begin{lemma}\label{lemma:zhe and maps}
Let $f,g,h,k$ be a $z$-homotopy equivalence relating $X,Y$ in $\CS$.  Let $(X,\a,z)$ be a curved twisted complex with curvature $z$, such that $\Id_X+\a\circ h$ is invertible in $\End^0_\CS(X)$, and define morphisms
\[
\begin{tikzpicture}[baseline=0em]
\tikzstyle{every node}=[font=\small]
\node (a) at (0,0) {$(X,\a,z)$};
\node (b) at (3,0) {$(Y,\b,z)$};
\path[->,>=stealth,shorten >=1pt,auto,node distance=1.8cm, thick]
(a) edge [loop left, looseness = 3,in=200,out=160] node[left] {$H$}	(a)
([yshift=3pt] a.east) edge node[above] {$F$}		([yshift=3pt] b.west)
([yshift=-2pt] b.west) edge node[below] {$G$}		([yshift=-2pt] a.east);
\end{tikzpicture}
\]
by the formulas
\begin{eqnarray*}
F&:=& f\circ (\Id_X+\a\circ h)\inv\\
G &:=& (\Id_X+h\circ \a)\inv \circ g\\
H&:=&h\circ (\Id_X+\a\circ h)\inv\\
\b&:=& zk+ f\circ \a\circ (\Id_X+h\circ \a)\inv\circ g
\end{eqnarray*}
Then $(Y,\b,z)$ is a well defined curved twisted complex, and $F,G$ are closed degree zero morphisms of curved twisted complexes with $d(H) = \Id_{(X,\a,z)} - G\circ F$.
\end{lemma}
Note that $\Id_X+\a\circ h$ is invertible if and only if $\Id_X+h\circ \a$ is invertible, since
\[
(\Id_X+h\circ \a) = \Id_X - h\circ (\Id_X + \a\circ h)\inv \circ \a.
\]
\begin{proof}
We omit the proof as it is a straightforward computation and actually follows from the proof of Theorem \ref{thm:curved hpl} below by setting $\e=0$.
\end{proof}

This appears very close to a proof that $(X,\a,z)$ and $(Y,\b,z)$ are homotopy equivalent and, indeed, in some cases (e.g.~if $z=0$; see Corollary \ref{cor:hpl without she}) this is already enough information to draw such a conclusion. However, to prove this formally requires an expression for the remaining piece of data: $K$.  Defining $K$ is a tricky matter.  Luckily, Markl has already told us how this is done in the case of zero curvature (spoiler alert: it requires the data of a strong homotopy equivalence).




\begin{theorem}[Curved homological perturbation lemma]\label{thm:curved hpl}
Let $z$ be a closed degree 2 element of the center of $\CS$, and let $X,Y\in \CS$ be given.  Let $\tilde{f}(\e)$, $\tilde{g}(\e),\tilde{h}(\e), \tilde{k}(\e)$ be the data of a strong homotopy equivalence $X\simeq Y$ in $\CS$.  Assume that $\CS$ is complete with respect to a two-sided ideal $\ng$ containing $z$.  Replacing $\e\mapsto z+\e$ yields morphisms
\[
f(\e):=\tilde{f}(z+\e),\qquad\quad g(\e):=\tilde{g}(z+\e),\qquad\quad h(\e):=\tilde{h}(z+\e),\qquad\quad k(\e):=\tilde{k}(z+\e).
\]
Assume that $(X,\a,z)$ is a curved twisted complex with $\a\in \End_\CS(X)\cap \ng$.  Define morphisms
\[
\begin{tikzpicture}
\tikzstyle{every node}=[font=\small]
\node (a) at (0,0) {$X$};
\node (b) at (3,0) {$Y$};
\path[->,>=stealth,shorten >=1pt,auto,node distance=1.8cm, thick]
(a) edge [loop left, looseness = 5,in=215,out=145] node[left] {$H(\e)$}	(a)
([yshift=3pt] a.east) edge node[above] {$F(\e)$}		([yshift=3pt] b.west)
([yshift=-2pt] b.west) edge node[below] {$G(\e)$}		([yshift=-2pt] a.east)
(b) edge [loop left, looseness = 5,in=35,out=325] node[right] {$K(\e)$}	(b);
\end{tikzpicture}
\]
in $\CS\llbracket\e \rrbracket$ by the formulas
\begin{subequations}
\begin{eqnarray}
F(\e) &:=& f(\e)\circ (1+\a \circ h(\e))\inv\nonumber\\
G(\e) &:=& (1+h(\e)\circ\a)\inv \circ g(\e)\nonumber\\
H(\e) & := & h(\e)\circ(1+\a \circ h(\e))\inv\nonumber\\
\beta +\e K(\e) &:=& (z+\e) k(\e) + f(\e)\circ \a\circ (\Id_X + h(\e)\circ \a)\inv \circ h(\e)\label{eq:beta def}
\end{eqnarray}
Then $F(\e)$, $G(\e)$, $H(\e)$, $K(\e)$ define a strong homotopy equivalence relating the curved twisted complexes $(X,\a,z_X)$ and $(Y,\b,z_Y)$.
\end{subequations}
\end{theorem}
Note that formal series such as
\[
\tilde{f}(z+\e) = \sum_{l\geq 0} (z+\e)^l \tilde{f}_l = \sum_{i\geq 0} \e^i\sum_{j\geq 0} z^j \tilde{f}_{i+j}
\]
are well-defined from our assumption of completeness.  Additionally, after setting $\e=0$, we obtain the data $f(0),g(0),h(0),k(0)$ of a $z$-homotopy equivalence $X\simeq Y$ in $\CS$, and furthermore the above expressions for $F(0),G(0),H(0)$ and $\b$ are as constructed in Lemma \ref{lemma:zhe and maps}.

Note also that $\b$ is intended to be an endomorphism of $Y$ in $\CS$ (equivalently an endomorphism in $\CS\llbracket\e \rrbracket$ with only constant term), so equation \eqref{eq:beta def} defines $\b$ to be the $\e^0$ coefficient of the right-hand side, with all higher $\e$-degree terms defining $\e K(\e)$.

\begin{proof}
For the proof we will abbrevate by writing $f=f(\e)$, $F=F(\e)$, and so on, and also writing $1=\Id$, and omitting all occurences of $\circ$.  We also abbreviate by introducing
\[
\Theta \ \ := \ \ \a (1+h \a)\inv.
\]
For the reader's convenience we recall all of the initial assumptions and definitions:
\begin{eqnarray*}
d(f) & =& (z+\e) (f h - kf)\\
d(g)&=&(z+\e)(gk-hg)\\
d(h)&=& \Id_X - gf-(z+\e)h^2\\
d(k) &=& \Id_Y- fg - (z+\e)k^2\\
d(\a)&=& z-\a^2\\
F &:=& f (1+\a h)\inv\\
G &:=& (1+h\a)\inv g\\
H &:=& h (1+\a h)\inv\\
\b+\e K &:=& (z+\e)k+ f \Theta g,
\end{eqnarray*}
\begin{subequations}
and we must show that
\begin{eqnarray}
d(F) +\b F -F\a& =& \e (FH - KF)\label{eq:df hat}\\
d(G)+\a G - G\b&=& \e (GK - HG)\label{eq:dg hat}\\
d(H)+\a H+H\a  &=& \Id_X - GF-\e H^2\label{eq:dh hat}\\
d(K)+\b K+K\b &=& \Id_Y- FG- \e K^2\label{eq:dk hat}\\
d(\b)+\b^2 & = & z\label{eq:dbeta}
\end{eqnarray}
Note that equations \eqref{eq:dk hat}, \eqref{eq:dbeta} are equivalent to 
\begin{equation}\label{eq:dk hat 2}
d(\b+\e K) + (\b+\e K)^2\  = \ z\Id_Y+\e (\Id_Y - FG).
\end{equation}
\end{subequations}
Many of the details of the below computations are easily verified but tedious, and we omit them.  Compute:
\begin{eqnarray*}
d(f) (1+\a h)\inv & =& (z+\e)(fH - kF) \\
fd((1+\a h)\inv) & = &  F\a- z fH - f\Theta g F - \e f H + \e F H
\end{eqnarray*}
Adding these gives
\[
d(F) = -((z+\e) k + f\Theta g) F + F \a + \e F H  = -(\b + \e K)F + F \a + \e F H,
\]
which proves \eqref{eq:df hat}.

Compute:
\begin{eqnarray*}
d((1+h\a)\inv) g &=&  G f\Theta g  - \a G    +z H g+\e H g- \e H G\\
(1+h\a)\inv) d(g) &=&  z G k - (z+\e) H g 
\end{eqnarray*}
Adding these yields
\[
d(G) =  -\a G + G((z+\e)k + f\Theta g) - \e H G=  -\a G + G(\b+\e K) - \e H G,
\]
which proves \eqref{eq:dg hat}.

Compute:
\begin{eqnarray*}
d(h) (1+\a h)\inv & =& 1 - \a H - (z+\e) h H\\
-h d((1+\a h)\inv) & =&  -H\a + z hH +  g F - GF+ \e h H -\e H^2
\end{eqnarray*}
Adding these gives $d(H) = 1-GF - \a H - H \a - \e H^2$, which proves \eqref{eq:dh hat}.

Finally, to prove \eqref{eq:dk hat 2}, one computes
\begin{eqnarray*}
(z+\e)d(k) & = & (z+\e)(1 - fg - (z+\e)k^2)\\
d(f) \Theta g & =& (z+\e) (fg- fG - k f \Theta g)\\
f d(\Theta) g & =& (z+\e) (f G + Fg - fg) - \e FG -  (f\Theta g)^2 \\ 
-f \Theta d(g) & =& (z+\e) (fg- Fg - f \Theta g k )\\
(\b+\e K)^2 & =& ((z+\e)k + f \Theta g)^2\nonumber\\
 & =& (z+\e)^2k^2 + (z+\e)(k f \Theta g + f\Theta g k) + (f \Theta g)^2
\end{eqnarray*}
Adding all of these gives \eqref{eq:dk hat 2} and completes the proof.
\end{proof}

\subsection{Corollaries}
\label{ss:consequences}
In this section we extract some useful corollaries of Theorem \ref{thm:curved hpl}.  
Many concrete applications of homological perturbation are instances of the following.

\begin{corollary}\label{cor:onesided replacement}
Let $(I,\leq)$ be a poset with the property that for each $i\in I$, there are finitely many $j\in I$ such that $i\leq j$.  Let $X_i\in \CS$ be objects indexed by $i\in I$, and suppose we have a curved twisted complex $(\bigoplus_i X_i,\a,z)\in \cTw(\CS)$ where $z$ is in the center of $\CS$ and $\a$ is strictly lower triangular with respect to the partial order on $I$ (i.e.~the component $\a_{ji}$ of $\a$ from $X_i$ to $X_j$ is zero for $i>j$).

Then if $X_i\simeq Y_i$ for some objects $Y_i\in\CS$ then there is a homotopy equivalence $(\bigoplus_i X_i,\a,z)\simeq (\bigoplus_i Y_i,\b,z)$ where $\b$ is strictly lower triangular.
\end{corollary}
\begin{proof}
Fix the poset $I$.  Define an auxilliary dg category $\CS'$ as follows.  Objects of $\CS'$ are objects in $\CS$ equipped with a direct sum decomposition $X = \bigoplus_{i\in I} X_i$, and morphism complexes in $\CS'$ are
\[
\Hom_{\CS'}(\bigoplus_{i\in I}X_i,\bigoplus_{j\in I}Y_j) = \prod_{i\in I}\Hom_\CS(X_i,\bigoplus_{j\geq i} Y_j) = \prod_{j\geq i\in I}\Hom_\CS(X_i, Y_j),
\]
where this latter isomorphism follows since the direct sum $\bigoplus_{j\geq i}Y_j$ is finite (for fixed $i$) by assumption.  That is to say, a morphism in $\CS'$ is a morphism in $\CS$ which is weakly lower triangular with respect to the given direct sum decompositions.  Let $\ng$ denote the two-sided ideal of $\CS'$ consisting of morphisms which are \emph{strictly} lower triangular.  Then $\CS'$ is complete with respect to $\ng$, and the rest of the proof is an application of Theorem \ref{thm:curved hpl}.
\end{proof}
\begin{remark}
There is a dual version, in which $\bigoplus$ is replaced by $\prod$, and the boundedness condition on $I$ is reversed.  The statement and its proof can be obtained by passing to the opposite category.
\end{remark}

The following is Markl's Ideal Perturbation Lemma \cite{MarklIdeal}, stated in our language.  Note that it is essentially the curvature zero version Theorem \ref{thm:curved hpl}, except that we do not assume given a 2-sided tensor ideal containing $\a$, with respect to which $\CS$ is complete.

\begin{corollary}[Markl's Lemma]\label{cor:uncurved hpl}
Suppose $f(\e),g(\e),h(\e),k(\e)$ are the data of a strong homotopy equivalence $X\simeq Y$ in $\CS$.  Let $\tw_\a(X)$ be a twisted complex such that $\Id_X+ \a \circ h_0$ is invertible. 
Define morphisms $\b \in \End_\CS(Y)$ and $F\in \Hom_{\CS\llbracket\e\rrbracket}^0(X,Y)$, $G\in  \Hom_{\CS\llbracket\e\rrbracket}^0(Y,X)$, $H\in  \End_{\CS\llbracket\e\rrbracket}^{-1}(X)$, $K\in \End_{\CS\llbracket\e\rrbracket}^1(Y)$, by the formulas
\begin{subequations}
\begin{eqnarray}
F(\e) &:=& f(\e)\circ (\Id_X+\a \circ h(\e))\inv\\
G(\e) &:=& (1+h(\e)\circ\a)\inv\circ g(\e)\\
H(\e) & := & h(\e)\circ(\Id_X+\a \circ h(\e))\inv\\
\beta +\e K(\e) &:=& \e k(\e) + f(\e)\circ \a\circ (\Id_X + h(\e)\circ \a)\inv \circ f(\e)
\end{eqnarray}
Then $F(\e)$, $G(\e)$, $H(\e)$, $K(\e)$ define a strong homotopy equivalence relating the curved twisted complexes $(X,\a,z_X)$ and $(Y,\b,z_Y)$.
\end{subequations}
\end{corollary}
\begin{proof}
From the assumptions, $\Id_X + \a \circ h(\e)$ and $\Id_X +h(\e)\circ \a$ are invertible elements of $\End_{\CS\llbracket \e\rrbracket}(X)$ since they are formal series in $\e$ with invertible constant terms.  Thus, the morphisms $F(\e)$, $G(\e)$, $H(\e)$, and $K(\e)$ are well-defined.  That they satisfy the desired relations follows from Theorem \ref{thm:curved hpl} (or its proof) with $z=0$.
\end{proof}

As was shown already, any homotopy equivalence $X\simeq Y$ gives rise to a strong homotopy equivalence.  It is often useful to bypass strong homotopy equivalences altogether, via the following.

\begin{corollary}\label{cor:hpl without she}
Let $f_0\in \Hom^0_\CS(X,Y)$, $g_0\in \Hom^0_\CS(Y,X)$, $h_0\in \End^{-1}_\CS(X)$, $k_0\in \End^{-1}_\CS(Y)$ be a homotopy equivalence relating $X,Y$ in $\CS$.  Let $\tw_\a(X)$ be a twist of $X$ such that $\Id_X+\a\circ h$ is invertible in $\End^0_\CS(X)$, and define morphisms
\[
\begin{tikzpicture}[baseline=0em]
\tikzstyle{every node}=[font=\small]
\node (a) at (0,0) {$\tw_\a(X)$};
\node (b) at (3,0) {$\tw_\b(Y)$};
\path[->,>=stealth,shorten >=1pt,auto,node distance=1.8cm, thick]
(a) edge [loop left, looseness = 3,in=200,out=160] node[left] {$H$}	(a)
([yshift=3pt] a.east) edge node[above] {$F$}		([yshift=3pt] b.west)
([yshift=-2pt] b.west) edge node[below] {$G$}		([yshift=-2pt] a.east);
\end{tikzpicture}
\]
by the formulas
\begin{eqnarray*}
F_0&:=& f_0\circ (\Id_X+\a\circ h_0)\inv\\
G_0 &:=& (\Id_X+h_0\circ \a)\inv \circ g_0\\
H_0 &:=&h\circ (\Id_X+\a\circ h_0)\inv\\
\b &:=& f_0\circ \a\circ (\Id_X+h_0\circ \a)\inv\circ g_0
\end{eqnarray*}
Then $\tw_\b(Y)$ is a well-defined twist of $Y$ and the maps $F_0,G_0$ are inverse homotopy equivalences $\tw_\a(X)\simeq \tw_\b(Y)$ and, in fact, the boundary of $H_0\in \End^{-1}(\tw_\a(X))$ is $\Id_{\tw_\a(X)}-G_0\circ F_0$.
\end{corollary}
The subscripts on the morphisms $f_0,F_0,\ldots$ above indicate that we are making statements about ordinary homotopy equivalences, not strong homotopy equivalences, and the relevant identities will be obtained from those involving s.h.e.~ by taking $\e=0$.
\begin{proof}
Let $f_0,g_0,h_0,k_0$ denote some chosen homotopy equivalence data $X\simeq Y$ in $\CS$. Modifying $k_0$ if necessary (as in the proof of Lemma \ref{lemma:cones detect equiv}) we may as well assume the existence of  $g_1\in \End^{-2}_\CS(Y)$ such that $d(g_1)+h_0g_0  -g_0k_0$.  Then $\smMatrix{-k_0 & f_0 \\ g_1 & h_0}$ defines a contracting homotopy for $\Cone(g_0)$.  From these initial data $f_0,g_0,h_0,k_0$ and $g_1$ we can construct the data of a strong homotopy equivalence $X\simeq Y$ via
\[
\sqmatrix{-k(\e) & f(\e) \\ c(\e) & h(\e)}  \ : = \ \sum_{l\geq 0} (-\e)^l C_l \sqmatrix{-k_0 & f_0 \\ g_1 & h_0}^{2l}
\]
and $g(\e):=g_0+\e c(\e)$, as Lemma \ref{lemma:she iff cone sc} and Lemma \ref{lemma:null implies strongnull}.  Then, from $f(\e),g(\e),h(\e),k(\e)$ we construct $\b\in \End^1_\CS(Y)$ as well as the data $F(\e),G(\e),H(\e),K(\e)$  of a strong homotopy equivalence $\tw_\a(X)\simeq \tw_\b(Y)$ as in Corollary \ref{cor:uncurved hpl}.   Setting $\e=0$, we obtain the desired relations involving $\b$, $F_0=F(0),$  $G_0=G(0),$ and  $H_0=H(0)$. 
\end{proof}

\printbibliography
\end{document}